\documentclass[12pt]{amsart}

\title{Zeta functions of orders on surfaces}

\author{Daniel Chan}
\thanks{The authors were supported by the Australian Research Council Discovery Project grant  DP220102861}
\address{School of Mathematics and Statistics, 
UNSW Sydney, 
NSW 2052,
Australia
}
\email{danielc@unsw.edu.au}

\author{Sean Lynch}
\address{Institute for Theoretical Sciences, Westlake University, Hangzhou, Zhejiang, China}
\email{sblynch@westlake.edu.cn}

\usepackage{mymacros,amssymb}
\usepackage{amsmath}
\usepackage{amsfonts}
\usepackage{amscd}
\usepackage{hyperref}

\thanks{}

\newcommand{\rad}{\operatorname{rad}}
\newcommand{\Hilb}{\operatorname{Hilb}}
\newcommand{\Serre}{\operatorname{Serre}}

\newtheorem{addendum}[theorem]{Addendum}
\newcommand{\Var}{\operatorname{Var}}
\newcommand{\mot}{\operatorname{mot}}

\begin{document}

\begin{abstract}
In this manuscript, we give effective methods for computing the zeta function of maximal orders on surfaces.  
\end{abstract}

\maketitle

\section{Introduction}

Zeta functions were classically studied for rings of integers and curves over finite fields. There are a number of ways to extend the definition to (sheaves of) noncommutative algebras $\cA$, the most interesting perhaps, and the one we shall focus on is 
\begin{equation}  \label{eq:introZeta}
    \zeta_{\cA}(s) : = \sum_{n \in \bN} a_n n^{-s}
\end{equation}
where $a_n$ is the number of left ideals $\cI$ with $|\cA/ \cI| = n$.

In this direction, Hey already studied zeta functions for maximal orders over rings of integers in her PhD thesis \cite{Hey}. Interestingly, the zeta function encodes information about the ramification data of the order which measures how far the order is from being Azumaya. It thus became an important tool in Weil's classic treatment of class field theory \cite{Weil}. Since then, there has been a lot of work on zeta functions of orders in the one-dimensional case, especially by Bushnell-Reiner \cite{BR80,BR85} in their work on Solomon's conjectures \cite{Solomon2}. Though the pair quickly solved the first conjecture, the second conjecture was finally fully resolved by Iyama \cite{Iyama03}.

For two-dimensional orders, much less work has been done. Berndt's survey \cite{Ber03} on K\"ahler's calcolo zeta program sums things up nicely. Segal \cite{Segal2,Segal} rediscovered much of this program; Fukshansky, Kühnlein and Schwerdt \cite{FKS17} rediscovered the same formula as Segal \cite{Segal}. Perhaps the most interesting developments in this case are for (essentially) commutative smooth projective surfaces $X$ over $\bF_q$ in G\"ottsche's study \cite{Gottsche} of Hilbert schemes of points $\Hilb^n(X)$ on surfaces. Since the $a_n$ can be computed essentially by counting $\bF_q$-rational points on $\Hilb^n(X)$, the Weil conjectures can be used to determine the Poincar\'e  polynomials for the $\Hilb^n(X)$ and his work gives beautiful formulae for ``Poincar\'e polynomial'' versions of the zeta function. These can be thought of as given by (\ref{eq:introZeta}) where the $a_n$ are now Poincar\'e polynomials of $\Hilb^n(X)$ (see Section~\ref{sec:global} for precise definitions). 

In this paper, we wish to study zeta functions of maximal orders $\cA$ on arithmetic surfaces $X$, i.e. surfaces over $\mathbb{F}_q$ or curves over $\mathbb{Z}$. Our study is guided by two key questions. Firstly, how is the ramification data of the order encoded in the zeta function if at all? Secondly, as a test case, can one develop effective computational methods, sufficient to calculate the zeta functions of noncommutative projective planes corresponding to a maximal order on $\bP^2_{\bF_q}$ ramified on a cubic curve? 

To this end, it is natural that one ought to impose some ``smoothness'' assumptions. Indeed, even for commutative curves, the theory of zeta functions is much more complicated, and for surfaces, most work has centred on Kleinian singularities and their stacky resolutions \cite{Naka}, \cite{GNS18}, where the formulas are quite complicated and involve substantial representation theory. In particular, we impose a stronger condition than finite global dimension. Perhaps the most natural notion of ``smoothness'' is that of terminal as introduced in the minimal model program for orders on surfaces in \cite{CI05}. Unfortunately, terminal orders in the arithmetic setting have only been classified in restricted settings \cite{CI24}.

Our starting point, will be a local formula for zeta functions established in the second author's PhD thesis \cite{Lyn}. In general, it seems that the formula is not appropriate for piecing together to obtain global zeta functions. However, the following special case seems to work surprisingly well.
\begin{theorem}  \label{thm:introSean}
Let $R$ be a commutative local noetherian normal domain of dimension two and $A$ an $R$-order. Suppose there exists a regular normal element $z \in \rad A$ such that $\bar{A} := A/(z)$ is isomorphic to a product of maximal orders over discrete valuation rings and that $A/\rad A \simeq \prod^m M_r(\bF_q)$. Then 
$$\zeta_A(s) = \prod_{n=0}^{\infty} \zeta_{\bar{A}}((n+1)s -n)
$$
where Hey's formula for zeta functions gives $\zeta_{\bar{A}}(s) = \prod_{j=0}^{r-1}(1-q^{j-rs})^{-m}$. 
\end{theorem}
This follows from Proposition~\ref{prop:cofiniteLengthIdealCond} and Corollary~\ref{cor:SeansLocalThm}. The important point here is that the formula is expressed as a relatively straightforward product and our first goal is to introduce a new type of Hecke operator which allows us to reprove this formula in a way which illuminates why we obtain such a product formula. This is accomplished in Section~\ref{sec:local}. 

Section~\ref{sec:global} looks at global zeta functions. Theorem~\ref{thm:introSean} applies quite readily to Azumaya algebras of $\cA$ of degree $d$ on a two-dimensional regular integral scheme $X$ of finite type over $\bZ$ to give in this case (see Proposition \ref{prop:zetaAzumaya})
$$\zeta_{\cA}(s)=\prod_{n=1}^\infty\prod_{j=1}^d \zeta^{\Serre}_X(nd(s-1)+j)$$
where $\zeta^{\Serre}_X$ is Serre's zeta function counting 0-cycles \cite{Ser65} (whose definition is recounted in Equation~(\ref{eq:defSerreZeta})). 
The rest of the section shows how the global zeta function encodes ramification, at least when the ramification is sufficiently nice that we know we can apply the local formula above. In particular we obtain
\begin{theorem}
Let $\cA$ be a maximal order of degree $d$ on a smooth surface over $\bF_q$, where $\bF_q$ contains a primitive $d$-th root of unity. Suppose the ramification consists of a single smooth ramification curve $Y$ and that the ramification is given by the degree $e$ cover $\tilde{Y} \to Y$. Then 
$$  \frac{\zeta_{\cA}(s)}{\zeta_{M_d(\cO_X)}(s)} = 
\frac{\prod_{n=1}^{\infty}\prod_{j=1}^{d/e} \zeta^{\Serre}_{\tilde{Y}}(\tfrac{d}{e} n(s-1) + j)}{\prod_{n=1}^{\infty}\prod_{j=1}^{d} \zeta^{\Serre}_{Y}(dn(s-1) + j)}.$$
\end{theorem}
This follows from Propositions~\ref{prop:zetaAzumaya} and \ref{prop:zetaSmoothRamGlobal}. 
Formulae for Poincar\'e zeta functions are also given. Section~\ref{sec:eg} looks at noncommutative analogues of the projective plane. Subject to some mild assumptions, the zeta functions have been computed. 

\begin{theorem}  \label{thm:introSklyanin}
Let $\cA$ be a maximal order of degree $d$ on $X=\bP^2_{\bF_q}$ that is ramified on a cubic curve $Y$ with ramification index $e$. Suppose that $q$ is large enough that it contains a primitive $d$-th root of unity, $Y$ has an $\bF_q$-rational point and that the singularities of $Y$ are $\bF_q$-rational. Suppose also that the ramification covers are totally ramified at the singularities of $Y$. Then 
$$  \frac{\zeta_{\cA}(s)}{\zeta_{M_d(\cO_X)}(s)} = 
\frac{\prod_{n=1}^{\infty}\prod_{j=1}^{d/e} \zeta^{\Serre}_{Y}(\tfrac{d}{e} n(s-1) + j)}{\prod_{n=1}^{\infty}\prod_{j=1}^{d} \zeta^{\Serre}_{Y}(dn(s-1) + j)}.$$
\end{theorem}
In the case that $Y$ is singular, $\zeta^{\Serre}_Y(s) = \frac{1}{(1-q^{1-s})^h}$ where $h=3$ when $Y$ is 3 lines, $h=2$ when $Y$ is a conic and a line and $h=1$ when $Y$ is a nodal cubic. A curious point here is that the zeta function does not depend on the particular ramification cover, just its degree. 

\section{Background on Zeta functions} \label{sec:backgroundZeta}

There are various types of zeta functions floating around, and in this section, we fix notation for the ones of primary interest in this paper. 

Let $\cA$ be an order on a noetherian normal integral scheme $X$. For us, this means that $\cA$ is a coherent sheaf of algebras on $X$ which is torsion-free and $\cA \otimes_{\cO_X} K(X)$ is a simple artinian algebra over the function field $K(X)$ of $X$. Its {\em degree} is defined to be $\deg A = \deg \cA \otimes_{\cO_X} K(X)$

We wish to investigate zeta functions  for such orders. The versions we use are the following. 
\begin{equation} \label{eq:defZetaFn}
\zeta_{\cA}(s) = \sum_{\text{cofinite}\ \cI \leq \cA} |\cA/\cI|^{-s}
\end{equation}
where the sum is over all left ideals $\cI$ such that $\cA/\cI$ is a coherent $\mathcal{O}_X$-module supported at a finite number of closed points with finite residue field so the cardinality $|\cA/\cI|$ is naturally defined. In this paper, we will view the sum in (\ref{eq:defZetaFn}) as a formal Dirichlet series, defined if for any $n \in \bN$, the number of ideals $\cI$ with $|\cA/\cI| = n$ is finite. Convergence questions will not be considered here.

 We will use $x \in X$ to mean that $x$ is a closed point of $X$ and let $\kappa(x)$ denote the residue field. Local zeta functions arise when we consider complete localisations $\hat{\cA}_x:= \cA \otimes_{\cO_{X}} \hat{\cO}_{X,x}$. The Chinese remainder theorem gives the standard product expansion
 \begin{equation}  \label{eq:localGlobalProductExpansion}
     \zeta_{\cA}(s) = \prod_{x \in X} \zeta_{\hat{\cA}_x}(s).
 \end{equation}

Often, $X$ will be defined over a finite field, say $\bF_q$ in which case we will set $t = q^{-s}$ and use
\begin{equation}  \label{eq:defZFn}
Z_{\cA}(t) = \sum_{\text{cofinite}\ \cI \leq \cA} t^{\dim_{\bF_q} \cA/\cI}
\end{equation}
so $\zeta_{\cA}(s) = Z_{\cA}(q^{-s})$.

If $\cA = \cO_X$, the zeta functions here will in general differ from Serre's zeta functions which can be defined as follows.
\begin{equation}  \label{eq:defSerreZeta}
\zeta_X^{\Serre}(s) = \prod_{x \in X} (1 - |\kappa(x)|^{-s})^{-1}, \quad
Z_X^{\Serre}(t)=\prod_{x\in X} (1-t^{\deg x})^{-1}
\end{equation}
where in the formula for $Z_X$ we are assuming that $X$ is defined over $\bF_q$ and $\deg x = [\kappa(x) : \bF_q]$. The difference between the two definitions is that the zeta functions we study count 0-dimensional subschemes (or their noncommutative analogues), whereas Serre's zeta function counts 0-cycles. Serre's zeta functions have been intensely studied, in particular, via the Weil conjectures. We will thus view them as ``known'' quantities and one of our key aims is to relate the zeta functions $\zeta_\cA$ to Serre's zeta functions of varieties related to $\cA$.

\section{Local zeta functions via Hecke operators}  \label{sec:local}

Given the local-global formula~(\ref{eq:localGlobalProductExpansion}), it is natural to study zeta functions for orders over complete local rings first. We thus revisit some formulae for zeta functions of local orders given in \cite{Lyn}. What is remarkable about these formulae, is that in many cases of interest, they are infinite products and so computationally, are nicely compatible with Equation~(\ref{eq:localGlobalProductExpansion}), as will be seen in later sections. Our goal in this section, is to introduce a new type of Hecke operator to reprove the local formulae and clarify why they are infinite products. 

Throughout this section, we assume that $R$ is a commutative local noetherian normal domain with finite residue field, and $A$ is an $R$-order in a central simple $K(R)$-algebra. As in \cite{Lyn}, we will look more generally at zeta functions of a finitely generated $A$-module $M$ (these are defined below in (\ref{eq:defMultiParameterZeta})). By default, we work with left modules. 

\begin{definition}  \label{def:homogSlice}
We say that a projective $A$-module $M$ has a {\em homogeneous slice} if there exists a regular normal element $z \in \rad(A)$ such that 
(*) 
every finite colength submodule of $\bar{M}:= M/(z)$ is isomorphic to $\bar{M}$. Here, $\rad(A)$ is the Jacobson radical of $A$. 
\end{definition}

The following enables us to generate important examples in the applications, where homogeneous slices exist.

\begin{proposition}  \label{prop:cofiniteLengthIdealCond}
Let $\bar{R}$ be a commutative noetherian local ring and $\bar{A}$ a finite $\bar{R}$-algebra. Then $\bar{A}$ is a product of maximal orders (in simple artinian rings) over discrete valuation rings if and only if every finite colength left ideal is free of rank 1. 
\end{proposition}
\begin{proof}
Suppose first that $\bar{A}$ is a maximal order in a simple artinian ring $Q$. From \cite[Corollary to Proposition~3.3]{AG}, we know $\bar{A}$ is a principal ideal ring so any left ideal $I = Ax$ for some $x$. If $I$ has cofinite length then $Qx = Q$ so $x$ must be regular. This proves the forward implication. 

Conversely, suppose every finite colength left ideal of $\bar{A}$ is free of rank 1. The Jacobson radical $\rad(\bar{A}) \triangleleft \bar{A}$ has cofinite length so is projective. It follows from \cite[Theorem~7.3.14]{MR} that $\bar{A}$ is a hereditary noetherian ring. Now any such ring is a product of hereditary artinian rings and hereditary noetherian prime rings  by \cite[Theorem~5.4.6]{MR}. Furthermore, these factors also satisfy condition~(*) of Definition~\ref{def:homogSlice}. This rules out the artinian case (see for example \cite[Theorem~5.4.7]{MR}).  We may thus assume that $\bar{A}$ is hereditary noetherian prime and not artinian. Replacing $\bar{R}$ with the centre of $\bar{A}$ we may assume that $\bar{R}$ is prime too. Now any non-zero left ideal of $\bar{A}$ has finite colength by \cite[Proposition~5.4.5]{MR}. It follows that the Krull dimension of $\bar{R}$ must be one. Suppose that $\bar{A} \subseteq \bar{B}$ for another order $\bar{B}$ in $K(\bar{R}) \otimes_{\bar{R}} \bar{A}$. Pick a non-zero $r \in \bar{R}$ so that $r\bar{B} \leq \bar{A}$ so say $r \bar{B} = \bar{A} x$ for some $x \in \bar{A}$. Computing endomorphisms in $K(\bar{R}) \otimes_{\bar{R}} \bar{A}$ we find 
$\bar{B} = x^{-1} \bar{A} x
$ so $\bar{A}$ is indeed maximal. 
\end{proof}

\begin{example}   \label{eg:AzumayaHomogSlice}
    Suppose that $R$ is a two-dimensional regular local ring with finite residue field, and suppose that $A$ is an Azumaya $R$-algebra. Pick any $z\in\rad(R)-\rad(R)^2$ so that $R/(z)$ is a discrete valuation ring. Then $A/(z)$ is Azumaya and hence a maximal order over $R/(z)$. By Proposition~\ref{prop:cofiniteLengthIdealCond}, $z$ defines a homogeneous slice for $A$. 
\end{example}

We return to the general setup and assume from now on that $M$ has a homogeneous slice given by $z$ as in Definition~\ref{def:homogSlice}. Suppose that $A$ has, up to isomorphism, $g$ simple modules. The Grothendieck group of the category of finite length $A$-modules is the free abelian group on $g$ generators, say $t_1, \ldots, t_g$ corresponding to the $g$ simples. We will write the Grothendieck group using multiplicative notation so that the positive cone $\cL_0$ corresponds to the set of monomials in the $t_j$. Below, we will need a slight variant of the zeta functions introduced in Section~\ref{sec:backgroundZeta} which will naturally be elements in $\widehat{\bZ \cL_0}:= \bZ[[t_1, \ldots, t_g]]$. For any finite length $A$-module $N$, let $\nu(N) \in \cL_0$ denote its class in the Grothendieck group. 

Let $\cX$ be the set of all finite colength left submodules $X \leq M$, $\bZ \cX := \oplus_{ X \in \cX} \bZ X$ and $\widehat{\bZ \cX}$ its completion, by which we mean the set of all formal (not necessarily finite) linear combinations. There is a well-defined linear map
\begin{equation}  \label{eq:defnNu}
\nu^{\ddagger}\colon \widehat{\bZ \cX}  \longrightarrow \widehat{\bZ \cL_0}\colon X \mapsto \nu(M/X).
\end{equation}
which takes {\em the universal zeta sum} $Z_M:= \sum_{X \in \cX}X$ to the {\em multi-parameter zeta function}
\begin{equation}  \label{eq:defMultiParameterZeta}
    Z_M(t_1,\ldots, t_g) := \nu^{\ddagger}(Z_M).
\end{equation}

Our approach to computing local zeta functions, as in \cite{Lyn} is to use the ``slice'' given by $z$. To this end, consider the $z$-adic filtration on $A$ and $M$. The associated graded ring $\gr A$ is isomorphic to the skew-polynomial ring $\bar{A}[z; \sigma]$ where $\sigma$ is the automorphism of $\bar{A}$ defined by conjugation by $z$. Note that $\sigma$ also permutes the simple $A$-modules and hence the $t_j$. Now $M$ is projective and $z$ regular, so $$\gr M = \bar{M} \oplus (z \otimes \bar{M}) \oplus (z^2 \otimes\bar{M}) \oplus \ldots$$
where $z^i\otimes \bar{M}:=z_iA\otimes_A\bar{M}$. The filtration on $M$ induces filtrations on any subquotient of $M$ and in particular, on any $X \in \cX$. We will also need the Grothendieck group of the category of finite length graded $\gr A$-modules. If $S_j$ is the simple $A$ and hence $\bar{A}$-module corresponding to $t_j$, then we obtain a simple $\gr A$-module $z^i \otimes S_j:=z^iA \otimes_A S_j$ which can alternately be viewed as a copy of $S_j$ living in degree $i$, but where the $\bar{A}$-action is skewed by $\sigma^i$. 
Let again $\nu$ denote the map to the Grothendieck group and, abusing notation, we let $z^i(t_j) := \nu(z^i \otimes S_j)$ and $\cL$ be the free commutative monoid generated by $z^i(t_j)$ for $i \in \bN, j = 1,\ldots ,g$. We have the polynomial ring $\bZ \cL$ as well as its completion, the power series ring $\widehat{\bZ \cL} := \varprojlim_n \bZ \cL^{\leq n}$ where $\cL^{\leq n}$ consists of all monomials of degree $\leq n$. We also have a well-defined map 
\begin{equation}
    \nu^\dagger \colon \widehat{\bZ \cX} \to \widehat{\bZ \cL} \colon X \mapsto \nu (\gr M/X)
\end{equation}
There is a natural monoid morphism $z\colon \cL \to \cL \colon z^i(t_j) \mapsto z^{i+1}(t_j)$ which is compatible with the notation for the generators for $\cL$. It also induces (topological) ring homomorphisms $z \colon \bZ \cL \to \bZ \cL$ and $z \colon \widehat{\bZ \cL} \to \widehat{\bZ \cL}$. The maps $\nu^{\ddagger}$ and $\nu^{\dagger}$ are related by 
\begin{equation}
    \nu^{\ddagger} = \rho \nu^{\dagger}
\end{equation}
where $\rho\colon \widehat{\bZ \cL} \to \widehat{\bZ \cL_0}\colon z^i(t_j) \mapsto \sigma^i(t_j)$. Note that $\rho$ corresponds to the restriction of scalars functor induced by the map $\bar{A} \to \gr A$. 

We can now introduce our versions of Hecke operators.
\begin{definition}  \label{def:HeckeOp}
Let $\bar{N}$ be a finite length $\bar{A}$-module. We define the {\em $z$-adic Hecke operator} $T_{\bar{N}} \colon \widehat{\bZ \cX} \to \widehat{\bZ \cX}$ to be the linear map which sends
$$
\cX \ni X \mapsto \sum_{Y<X,\ Y \cap zM = zX, \ \gr(X/Y)_0 \simeq \bar{N}} Y.
$$
We also introduce the operator $T^- \colon \widehat{\bZ \cX} \to \widehat{\bZ \cX} \colon Y \mapsto z^{-1}(Y \cap zM)$ which sends any summand of $T_{\bar{N}} X$ back to $X$. Finally, consider the {\em $z$-adic zeta Hecke operator} $T := \sum_{\bar{N}} T_{\bar{N}}$ where $\bar{N}$ runs over all isomorphism classes of finite length $\bar{A}$-modules. 
\end{definition}

It is worthwhile unravelling the conditions defining the summation above. 

\begin{proposition}  \label{prop:grIgrJ}
Let $Y \in \cX$ and $X = T^-Y = z^{-1}(Y \cap zM)$. Then
\begin{enumerate}
    \item $Y \leq X$.
    \item If $\bar{Y}_i \leq \bar{M}$ are the left submodules such that $$\gr Y = \bar{Y}_0 \oplus (z \otimes \bar{Y}_1) \oplus \ldots$$ then $$\gr X = \bar{Y}_1 \oplus (z \otimes \bar{Y}_2) \oplus \ldots.$$
\end{enumerate}
\end{proposition}
\begin{proof}
Part~(1) follows since $zY  \leq Y \cap zM  = zX$ and part~(2) follows since 
$ \gr zX = \gr (Y \cap zM) = (z \otimes \bar{Y}_1) \oplus (z^2 \otimes \bar{Y}_2) \oplus \ldots
$.
\end{proof}

The next result shows how the $z$-adic zeta operator $T$ can be used to study the universal zeta sum. Below, we let $T^{-\;n}$ denote the $n$-th power of $T^-$, and caution that this is not an inverse to $T^n$.  
\begin{proposition} \label{prop:TandZetaSum}
\begin{enumerate}
    \item Let $X \in \cX$ and $n \in \bN$ be any integer such that $\gr(M/X)_n = 0$ (note such $n$ exist). Then $T^{-\;n} X = M$. 
    \item If $Y\in \cX$ and $X = T^-Y$, then $Y$ is a summand of $T_{\bar{N}} X$ where $\bar{N} = \gr(X/Y)_0$ and is not a summand of any other $T_{\bar{N}'}X'$. 
    \item $T^n M = \sum_{T^{-n}X = M} X$.
\end{enumerate}
\end{proposition}
\begin{proof} Part (1) follows from Proposition~\ref{prop:grIgrJ}(2) and the fact that $M/X$ and hence $\gr(M/X)$ have finite length. Part~(2) follows from Proposition~\ref{prop:grIgrJ}(1) and the definition of the Hecke operator. Finally, part~(3) follows from part~(2) by induction. 
\end{proof}

To use this proposition, we need to track how $T$ operates on $\widehat{\bZ \cL}$ as per the following definition.

\begin{definition}  \label{def:degree}
Let $U \colon \widehat{\bZ \cX} \to \widehat{\bZ \cX}$ be a linear map. We say that $U$ {\em has a degree map $\deg U$} if there is a continuous linear map $\deg U \colon \widehat{\bQ \cL } \to \widehat{\bQ \cL}$ making the diagram below commute.
$$\begin{CD}
\widehat{\bZ \cX} @>U>> & \widehat{\bZ \cX}\\
@V{\nu^{\dagger}}VV & @VV{\nu^{\dagger}}V \\
\widehat{\bQ \cL } @>{\deg U}>>  & \widehat{\bQ \cL}
\end{CD}$$
\end{definition}

We need some lemmas to show that the $z$-adic Hecke operators have degree maps and to compute them. 
\begin{lemma}  \label{lem:chiMultiplicative}
The function that assigns to a finite length $\bar{A}$-module $\bar{N}$ the cardinality $|\Hom_{\bar{A}}(\bar{M},\bar{N})|=: \chi_{\bar{M}}(\bar{N})$, descends to a multiplicative map of monoids $\chi_{\bar{M}}\colon \cL_0 \to \bN$. 
\end{lemma}
\begin{proof}
Our homogeneous slice hypothesis ensures that $\Hom(\bar{M},-)$ is exact and cardinality sends extensions of groups to products. 
\end{proof}

We need to count the summands which appear in any $T_{\bar{N}}X$. We use the argument in \cite[Lemma~3.4.4, Proposition~3.5.4]{Lyn} for a similar problem. 
\begin{lemma}  \label{lem:countingLifts}
Let $X \in \cX$ and $\bar{N}$ be a finite length $\bar{A}$-module. Let $\gr X = \bar{X}_0 \oplus (z \otimes \bar{X}_1) \oplus \ldots$ and $a_{\bar{N}}$ be the number of submodules $\bar{Y}_0$ of $\bar{X}_0$ such that $\bar{X}_0/\bar{Y}_0 \simeq \bar{N}$. Then the number of summands in $T_{\bar{N}}X$ is
$$a_{\bar{N}} \prod_{i \geq 1} \chi_{\bar{M}}(z^i\otimes (\bar{X}_{i}/\bar{X}_{i-1}))
$$
\end{lemma}
\begin{proof}
Pick $\bar{Y}_0\leq \bar{X}_0$ as above. Consider the diagram below whose row is a short exact sequence and the vertical map is inclusion.
$$\begin{CD}
     @. @. @. \bar{Y}_0 @. \\
    @.  @. @.  @VV{\iota}V  \\
    0 @>>>  \frac{X \cap zM}{zX}@>>> \frac{X}{zX}@>>> \frac{X}{X \cap zM} = \bar{X}_0@>>> 0
\end{CD}
$$
Note first that by our homogeneous slice hypothesis, $\bar{X}_0$ is projective so the short exact sequence splits. Let $Y$ be a summand of $T_{\bar{N}}X$ corresponding to $Y \leq X$ with $(\gr Y)_0 = \bar{Y}_0$. Note that $Y \geq zX$ so they are characterised by the submodules $Y/zX \leq X/zX$ whose image in $\bar{X}_0$ is $\bar{Y}_0$ and intersection with $\frac{X \cap zM}{zX}$ is trivial. These are in one to one correspondence with lifts of $\iota$ to $\tilde{\iota} \colon \bar{Y}_0 \to X/zX$, the correspondence given by sending $\tilde{\iota}$ to its image. Such lifts exist since the row splits, and fixing any particular one, all others can be obtained by adding an arbitrary homomorphism in $\Hom(\bar{Y}_0,\frac{X \cap zM}{zX})$. Our homogeneous slice hypothesis ensures that $\bar{Y}_0 \simeq \bar{M}_0$ so the number of summands we are counting is thus $a_{\bar{N}}\chi_{\bar{M}}(\frac{X \cap zM}{zX})$. By Lemma~\ref{lem:chiMultiplicative}, we may replace $\frac{X \cap zM}{zX}$ with its associated graded. Since $\gr (X \cap zM) = (z \otimes \bar{X}_1) \oplus (z^2 \otimes \bar{X}_2) \ldots$ and $\gr zX = z \otimes \gr X$, we are done. 
\end{proof}

\begin{notation}  \label{notn:Xi}
Given any monomial $m \in \cL$ we write $m = m_0m_{>0}$ where $m_0$ is the product of all the generators $t_1, \ldots,t_g$ appearing in $m$, and $m_{>0}$ is the product of all the other generators $z^i(t_j), i >0$. Recall our monoid morphism $z \colon \cL \to \cL$. Consider the multiplicative map
$$ \Xi \colon \cL \to \bQ \cL \colon m \mapsto \frac{\chi_{\bar{M}}(z(m))}{\chi_{\bar{M}}(m_{>0})} m_0 z(m)
$$
which extends to a continuous ring homomorphism $\Xi \colon \widehat{\bQ \cL} \to \widehat{\bQ \cL}$. 
\end{notation}

Note that $\bar{A}$ has the ``same'' simples as $A$ so the zeta function of the $\bar{A}$-module $\bar{M}$ is also naturally a power series in $t_1, \ldots, t_g$ too. 
\begin{proposition}  \label{prop:degT}
The $z$-adic Hecke operator $T_{\bar{N}}$ has a degree given by $\deg T_{\bar{N}} = a_{\bar{N}} \nu(\bar{N}) \Xi$ where $a_{\bar{N}}$ is the number of submodules $\bar{Y}$ of $\bar{M}$ with $\bar{M}/\bar{Y} \simeq \bar{N}$. Hence $\deg T = Z_{\bar{M}}(t_1,\ldots,t_g)\Xi$.
\end{proposition}
\begin{proof}
The second statement follows from the first so we prove the formula for $\deg T_{\bar{N}}$. Let $X \in \cX$ and $Y$ be a summand of $\deg T_{\bar{N}} X$. Then from Proposition~\ref{prop:grIgrJ}(2) and the definition of $T_{\bar{N}} X$, we know that $\nu^{\dagger}(Y) = \nu(\bar{N})m_0 z(m)$ where $m = \nu^{\dagger}(X)$. Finally, note that $(\gr X)_0 \simeq \bar{M}$ by our homogeneous slice hypothesis, so $a_{\bar{N}}$ defined above is the same as the $a_{\bar{N}}$ defined in Lemma~\ref{lem:countingLifts} counting the number of summands of $T_{\bar{N}} X$. That lemma completes the proof. 
\end{proof}

\begin{theorem}\label{thm:Sean}
Suppose the finitely generated projective $A$-module $M$ has a homogeneous slice given by $z \in \rad(A)$. Then with the notation above
$$ Z_M(t_1,\ldots,t_g) = \prod_{j=0}^{\infty} Z_{\bar{M}}(\Xi_j(t_1),\ldots, \Xi_j(t_g)).
$$
where $\Xi_j(t_i):= \rho \Xi^j(t_i)$.
\end{theorem}
\begin{proof}
Let us abbreviate $(t_1,\ldots, t_g)$ to $\vec{t}$. We use Proposition~\ref{prop:degT} and the fact that $\Xi$ is a continuous ring homomorphism to see
\begin{multline*}
\nu^{\ddagger} (T^n M) = \rho (\deg T)^n 1 = \rho (Z_{\bar{M}}(\vec{t}) \Xi)^n 1 \\ 
= \rho Z_{\bar{M}}(\vec{t})\Xi(Z_{\bar{M}}(\vec{t})) \ldots \Xi^{n-1}(Z_{\bar{M}}(\vec{t})) \\
= Z_{\bar{M}}(\vec{t})Z_{\bar{M}}(\Xi_1(\vec{t}))\ldots Z_{\bar{M}}(\Xi_{n-1}(\vec{t}))
\end{multline*}
Taking the limit as $n \to \infty$ completes the proof. 
\end{proof}

\begin{remark}
The above theorem is a special case of \cite[Theorem~3.6.1]{Lyn} which can also be readily proved as above, using variants of the $z$-adic Hecke operator. We will have no need of this generalisation which requires introducing quite a bit more notation.
\end{remark}

We specialise now to local zeta functions for orders, our objects of primary interest. 

\begin{corollary}
\label{cor:SeansLocalThm}
Let $A$ be an $R$-order with a homogeneous slice given by $z \in \rad(A)$ and suppose that $A/\rad (A) \simeq \prod^m M_r(\bF_q)$. Then
$$\zeta_A(s) = \prod^{\infty}_{n=1} \prod_{j=1}^r (1 - q^{nr-j}q^{-nrs})^{-m} = \prod^{\infty}_{n=1} \prod_{j=1}^r (1 - q^{-j}(qt)^{nr})^{-m}
$$
where $t = q^{-s}$. 
\end{corollary}
\begin{proof}
We use the notation above so $\bar{A}:=A/(z)$. Now by the generalised version of Hey's formula given in \cite[Theorem~2.3.3]{Lyn}, the homogeneous slice condition ensures
\begin{equation}  \label{eq:Hey}
 Z_{\bar{A}}(t_1,\ldots,t_m) = \prod_{i=1}^m \prod_{j=0}^{r-1} (1-q^j t_i)^{-1}.  
\end{equation}
Alternatively, one can use Proposition~\ref{prop:cofiniteLengthIdealCond} and the Wedderburn decomposition of $A/\rad{A}$, to see that $\bar{A}$ is the product of $m$ maximal orders. Then (\ref{eq:Hey}) is just Hey's formula, c.f. \cite[Eqn.~3.5]{BR87}. 

In this case $\chi_{\bar{A}}$ is just the cardinality function so $\chi_{\bar{A}}(z^j(t_i)) = |S_i| = q^r$. It follows by elementary induction that $\Xi^n(t_i) = q^{rn} t_i z(t_i) \ldots z^n(t_i)$. Let $\rho' \colon \bZ [[t_1,\ldots t_m]] \to \bZ [[t]]$ be the continuous homomorphism which maps $t_i \mapsto t^r$. Then by Theorem~\ref{thm:Sean}, we have 
$$ Z_{A}(t) = \rho' \rho \prod_{n=0}^{\infty} Z_{\bar{A}}(\Xi^n(t_1),\ldots,\Xi^n(t_m)) = 
\prod_{n=0}^{\infty} Z_{\bar{A}}(q^{rn}t^{r(n+1)},\ldots,q^{rn}t^{r(n+1)})
$$
since $\rho'\rho(\Xi^n(t_i)) = q^{rn}t^{r(n+1)}$. Changing product indices gives the formula in the Corollary. 
\end{proof}

We now investigate examples of orders which have homogeneous slices. We first recall the definition of symbols.

\begin{definition} \label{def:symbol}
Suppose $\xi \in R$ is a primitive $e$-th root of unity. Given $a,b \in R-0$ we define the {\em $R$-symbol} $(a,b):=(a,b)^R_{\xi}$ to be the $R$-algebra
$$
\Delta = \frac{R\langle x,y \rangle}{(x^e-a, y^e-b, yx-\xi xy)}
$$
and refer to $x,y$ as the {\em defining generators}. 
\end{definition}
We also need the following
\begin{notation}  \label{notn:defDeltad}
For any ring $\Delta$, $z \in \Delta$ and $l\in\mathbb{N}$ we let 
\begin{equation*}  
\Delta_l = \Delta_l(z) := 
\begin{pmatrix}
\Delta & \Delta & \cdots & \Delta \\
 z\Delta & \Delta & & \vdots \\
\vdots & \ddots & \ddots & \vdots \\
z \Delta & \cdots & z \Delta & \Delta
\end{pmatrix}
\subseteq M_l(\Delta).    
\end{equation*}
\end{notation}

\begin{proposition}  \label{prop:symbolsHomogSlice}
Consider an $R$-symbol $\Delta:=(a,b)^R_{\xi}$ where $R$ is complete regular local of dimension two with maximal ideal $\frakm$. Let $x,y$ be the corresponding defining generators and assume that the residue characteristic of $R$ is relatively prime to $e$, where $\xi\in R$ is a primitive $e$-th root of unity. Suppose that one of the following two occurs:
\begin{enumerate}
    \item $a,b$ is a regular system of parameters;
    \item $a \in \frakm - \frakm^2$ and $b \in R^{\times}$.
\end{enumerate}
Then $A:=M_r(\Delta_l(x))$ has a homogeneous slice and $A/ \rad A \simeq \prod^m M_r(\bF_q)$ for some $m,q$. 
\end{proposition}
\begin{proof}
If 
$$
\tilde{x} := 
\begin{pmatrix}
0 & 1 & \cdots & 0 \\
\vdots & 0 & \ddots&  \\
0 & \ddots & \ddots & 1 \\
x & 0 & \cdots & 0
\end{pmatrix}
$$
then $\tilde{x}\in\rad(\Delta_l(x))$ is a normal regular element of $\Delta_l(x)$ with 
$$\Delta_l(x)/(\tilde{x}) \simeq \prod^l \frac{(R/(a))[y]}{(y^e - b)}.$$  
By hypothesis, $\frac{(R/(a))[y]}{(y^e - b)}$ is a cyclic extension of the discrete valuation ring $R/(a)$ and is thus a product of isomorphic discrete valuation rings. Taking matrix rings and applying Proposition~\ref{prop:cofiniteLengthIdealCond}, we see that $A$ has a homogeneous slice and $A/\rad A$ has the desired form.
\end{proof}

Recall that ramification data for an order measures the failure for it to be Azumaya, just as the ramification theory for finite maps measures failure from being \'etale. The reader unfamiliar with this notion can look at \cite[Section~2]{CI24} for definitions and a summary of the basic theory. We content ourselves here with reviewing the ramification data for the orders arising in Proposition~\ref{prop:symbolsHomogSlice} since these are the ones that we will restrict our attention to. The order $A:= M_r(\Delta_l(x))$ is not Azumaya along the {\em ramification} curve $C$ defined by $a=0$ since $A/(a)$ has a non-trivial nilpotent ideal generated by the normal element $\tilde{x}$. The ramification data there is obtained by looking at the ``residue field extension''. More precisely, the ramification along $C$ is given by the {\em ramification cover}
$$ Z(A/(\tilde{x})) = \prod^l \frac{(R/(a))[y]}{(y^e -b)}.
$$
This gives a degree $el$ cover of $C$ so the {\em ramification index of $A$ along $C$} is $el$. In case~(2) of the Proposition, $C$ is the only ramification curve, whilst in case~(1), $b=0$ defines the only other ramification curve and the ramification cover there is given by $\frac{(R/(b))[x]}{(x^e -a)}$. The corresponding ramification index is $e$ and the ramification cover is totally ramified at the closed point. In particular, in case~(1), the order $M_r(\Delta_l(x))$ has the ramification data appearing in part~(3) of the Corollary below. 

\begin{corollary}  \label{cor:smoothRamLocal}
Let $\cA$ be a degree $d$ maximal order on a two-dimensional regular integral scheme $X$ of finite type over $\bZ[\frac1{d}]$ with ramification locus $Y \subset X$. Suppose $H^0(\cO_X)$ contains a primitive $d$-th root of unity and $p \in X$ is a closed point where one of the following holds
\begin{enumerate}
    \item $p \notin Y$
    \item $p$ is a regular point of $Y$ or, 
    \item $p$ is a normal crossing point of $Y$ and the ramification covers totally ramify above $p$. 
\end{enumerate}
Then the complete local ring $\hat{\cA}_p$ has a homogeneous slice. Furthermore, $\hat{\cA}_p/ \rad \hat{\cA}_p \simeq \prod^m M_r(\bF_q)$ for some $m,q$ and $r=d$ in case (1), and $r=d/e$ in cases (2) and (3) where $e$ is the ramification index of $\cA$ along the ramification curves through $p$. 
\end{corollary}
\begin{proof}
In case~(1), $\cA$ is Azumaya locally at $p$ so we are done by Example~\ref{eg:AzumayaHomogSlice} and the fact that the Brauer group of a finite field is trivial. Suppose we are in case~(2) or (3). 
Since $\cA$ is maximal, $\hat{\cA}_p$ is normal in the sense of \cite[Definition~2.1]{CI24} so \cite[Propositions~3.3 and 3.5]{CI24} shows that it has the form given in Proposition~\ref{prop:symbolsHomogSlice}. Note that $\hat{\mathcal{O}}_{X,p}$ contains a primitive $d$-th root of unity because $H^0(\mathcal{O}_X)$ does. Moreover, these $d$-th roots of unity remain primitive in the residue field $k(p)$ because the residue characteristic is relatively prime to $d$.
\end{proof}

\begin{remark}
It would be interesting to know how generally the formula for zeta functions in Corollary~\ref{cor:SeansLocalThm} holds. More importantly, what nice sufficient criteria, such as those given in Corollary~\ref{cor:smoothRamLocal}, are there, to ensure the formula holds. One natural candidate is that it should hold for all orders on surfaces which are in some sense smooth. The class of terminal orders is one such class that has been studied significantly (see \cite{CI05}). They are defined by ramification data and include the examples in Corollary~\ref{cor:smoothRamLocal}. Over an algebraically closed field, the \'etale local structure of terminal orders on surfaces has been completely classified and one always obtains rings of the form $M_r(\Delta_l(x))$ as in Proposition~\ref{prop:symbolsHomogSlice}. Unfortunately, this classification has not been completed in the arithmetic case, but all those studied so far (see \cite{CI24}) also always have a homogeneous slice. 
\end{remark}

\section{Global zeta functions}  \label{sec:global}

Let $X$ be a normal integral scheme of finite type over $\bZ$ and $\cA$ be an order on $X$. We use the local zeta function formula of Corollary~\ref{cor:SeansLocalThm} and the product expansion (\ref{eq:localGlobalProductExpansion}) to compute global zeta functions. We consider the Azumaya case first, which can be easily dealt with. 

\begin{proposition}  \label{prop:zetaAzumaya}
Let $\cA$ be a degree $d$ Azumaya algebra on a two-dimensional regular integral scheme $X$ of finite type over $\bZ$. Then 
$$\zeta_{\cA}(s)=\prod_{n=1}^\infty\prod_{j=1}^d \zeta^{\Serre}_X(nd(s-1)+j).$$
In particular, if $X$ is a smooth surface over $\bF_q$, then 
\begin{equation} \label{eq:zetaAzumaya}
    Z_{\cA}(t) = \prod^{\infty}_{n=1} \prod_{j=1}^d Z^{\Serre}_X(q^{-j}(qt)^{dn}).
\end{equation}
\end{proposition}
\begin{proof}
For any $x \in X$, $\hat{\cA}_x$ is Azumaya and we can apply Corollary~\ref{cor:SeansLocalThm} noting that $\hat{\cA}_x/\rad \hat{\cA}_x \simeq M_d(\kappa(x))$ since $\text{Br}\,\kappa(x) = 0$. We thus find 
$$\zeta_{\hat{\cA}_x}(s) = \prod^{\infty}_{n=1} \prod_{j=1}^d (1 - |\kappa(x)|^{nd(1-s)-j})^{-1}.$$
\end{proof}

As in Proposition \ref{prop:zetaAzumaya}, let $\mathcal{A}$ be an Azumaya algebra of degree $d$ on a two-dimensional regular integral scheme $X$ of finite type over $\mathbb{Z}$. It is natural to compare the above formula to the zeta function for the Brauer-Severi scheme $B$ of $\mathcal{A}$. We recall that the latter is the Hilbert scheme of $\cA$-module quotients of $\cA$ that are locally free of rank $d$ as a coherent sheaf on $X$. We compute $\zeta^{\Serre}_B$ as follows.  
For each closed point $x\in X$, we see that the fibre $B_x\simeq \mathbb{P}_{k(x)}^{d-1}$ over $k(x)$ because $B_x$ is a Brauer-Severi variety of dimension $d-1$ over the finite field $k(x)$. Therefore,
$$\zeta_{B}^{\Serre}(s)=\prod_{x\in X}\zeta_{\mathbb{P}_{k(x)}^{d-1}}^{\Serre}(s)=\zeta^{\Serre}_{\mathbb{P}^{d-1}_X}(s)=\prod_{j=1}^d\zeta_{X}^{\Serre}(s-d+j).$$
Hence, by Proposition \ref{prop:zetaAzumaya}, we have the formula
\begin{equation}    \label{eq:zetaBS}
\zeta_\mathcal{A}(s)=\prod_{n=1}^\infty \zeta_{B}^{\Serre}(nd(s-1)+d).
\end{equation}
We interpret this formula as follows. Points of $B$ correspond to ideals $\cI < \cA$ of colength one and hence, only a part of the data encoded in the zeta function for $\cA$. This is the $n=1$ factor in (\ref{eq:zetaBS}) which is $\zeta^{\Serre}_B(ds)$. The change of variables from $s$ to $ds$ can be viewed as follows. Given a $k(x)$-rational point of $B_x$, the corresponding ideal $\cI < \cA$ has quotient $\cA/\cI \simeq k(x)^d$. 

\begin{example}
    Let $R$ be the ring of integers in a global field $K$, let $X=\bA_{R}^1$, and let $X_0=\Spec R$. Then $\zeta_{X_0}^{\Serre}(s)=\zeta_X^{\Serre}(s+1)$ is just the Dedekind zeta function of $K$. Taking $\cA=\cO_X$ in Proposition \ref{prop:zetaAzumaya}, we recover Segal's formula \cite{Segal}
    $$\zeta_{\cA}(s)=\prod_{n=1}^\infty \zeta_X^{\Serre}(n(s-1)+1)=\prod_{n=1}^\infty \zeta_{X_0}^{\Serre}(n(s-1)).$$
\end{example}


The local-global product expansion for zeta functions in Equation~(\ref{eq:localGlobalProductExpansion}) naturally lends itself to the following definition. 
For any locally closed $Y\subseteq X$, let 
\begin{equation} \label{eq:defZetaStrata}
\zeta_{\cA,Y}(s) = \prod_{y \in Y} \zeta_{\hat{\cA}_y}(s).
\end{equation}

\begin{proposition}   \label{prop:zetaSmoothRamGlobal}
Let $\cA$ be a degree $d$ maximal order on a two-dimensional regular integral scheme $X$ of finite type over $\bZ[\frac{1}{d}]$ such that $H^0(\mathcal{O}_X)$ contains a primitive $d$-th root of unity. Suppose that $Y \subset X$ is a smooth ramification curve, crossing all other ramification curves normally, say at the points $y_1, \ldots, y_h$. Let $\pi\colon \tilde{Y} \to Y$ denote the ramification cover of $\cA$ along $Y$ and $e$ be the ramification index. Suppose furthermore, that all ramification covers are totally ramified above $y_1,\ldots, y_h$. Then 
$$ \zeta_{\cA,Y}(s) = \prod_{n=1}^{\infty}\prod_{j=1}^{d/e} \zeta^{\Serre}_{\tilde{Y}}(\tfrac{d}{e} n(s-1) + j).
$$
\end{proposition}
\begin{proof}
Let $y \in Y$ and note that by Corollary~\ref{cor:smoothRamLocal}, $A := \hat{\cA}_y$ has a homogeneous slice. Since ramification commutes with complete localisation, we see from Corollary~\ref{cor:smoothRamLocal} that for $y \notin \{y_1,\ldots,y_h\}$,
\begin{equation}  \label{eq:zetaSmoothRamGlobalTop}
 A / \rad A \simeq \prod_{\tilde{y} \in \pi^{-1}(y)} M_r(\kappa(\tilde{y}))
\end{equation}
where $r = d/e$. Actually, (\ref{eq:zetaSmoothRamGlobalTop}) also holds if $y=y_i$ since in that case $\pi^{-1}(y_i)$ consists set-theoretically of a single $\kappa(y_i)$-rational point. 
From Corollary~\ref{cor:SeansLocalThm}, we see that\begin{equation}
\zeta_A(s) = \prod_{\tilde{y} \in \pi^{-1}(y)} \prod_{n=1}^{\infty} \prod_{j=1}^{r} (1 - |\kappa(\tilde{y})|^{-nr(s-1)-j})^{-1}.
\end{equation}
Taking the product over all $y \in Y$ gives the desired formula. 
\end{proof}

Suppose now that $X$ is a quasi-projective variety over a field. We can think of left ideals in $\cA$ as points of the Hilbert scheme of quotients of $\cA$. For the reader unfamiliar with the theory of Hilbert schemes in the noncommutative context, the theory has been worked out rather generally in \cite{AZ01}. In particular, it is known that there exists a scheme $\Hilb^n(\cA)$ parameterising colength $n$ left ideals of $\cA$. Its existence can also be deduced from the corresponding Hilbert scheme of $\cO_X$-module quotients $\cA/\cF$ of the $\cO_X$-module $\cA$ and noting that the condition that $\cF$ is closed under left multiplication by $\cA$ is closed. We let $\Hilb^n_{red}(\cA)$ be the reduced induced subscheme of $\Hilb^n(\cA)$. 

In the commutative case $\cA = \cO_X$, the finite colength ideals correspond to the Hilbert scheme of points studied by G\"ottsche \cite{Gottsche} via the Weil conjectures. We follow his treatment closely in pursuing an analogous study for orders on surfaces. 

We will use the Weil conjectures in the following way. Let $X$ now be a variety over the finite field $\bF_q$. Then by Dwork's theorem on the rationality of the zeta function, there are algebraic integers $\alpha_i\in \bC^{\times}$ and integers $l_i$ such that 
\begin{equation}  \label{eq:WeilConj}
|X(\bF_{q^n})| = \sum_{i} l_i \alpha_i^n, \quad \text{for } \ n \in \bZ_{>0}.
\end{equation}
It will be convenient to let $\cW$ denote the monoid of geometric progressions of the form $(\alpha^n)_{n \in \bZ_{>0}}$ for some $\alpha \in \bC^{\times}$. This is of course naturally isomorphic to $\bC^{\times}$, but it allows us to view the sequence $(|X(\bF_{q^n})|)_{n \in \bZ_{>0}}$ as an element of the group algebra $\bQ \cW$ since geometric progressions are linearly independent. Consider the group homomorphism $\delta \colon \cW \to \bR\colon \alpha \mapsto 2\log_q|\alpha|$. We write the group algebra of $\bR$ as $\bQ[z^i]_{i \in \bR}$. We can lift $\delta$ to a group algebra homomorphism
\begin{equation}  \label{eq:defNu}
\nu \colon \bQ \cW \to \bQ[z^i]_{i \in \bR} \colon (\alpha^n)_{n \in \bZ_{>0}} \mapsto z^{2 \log_q|\alpha|}.
\end{equation}
If $X$ is smooth and projective, then the Riemann hypothesis part of the Weil conjectures means that $\nu(|X(\bF_{q^n})|)_{n \in \bZ_{>0}} \in \bQ[z]$. 

\begin{definition} \label{def:WeilBetti}
Suppose that $X$ is a quasi-projective variety over $\bF_q$ with the property that $\nu(|X(\bF_{q^n})|)_{n \in \bZ_{>0}} \in \bQ[z]$. Then we say that $X$ {\em has a Weil-Poincar\'e polynomial} $P(X,z)$ and it is defined by 
$$ P(X, -z) := \nu (|X(\bF_{q^n})|)_{n \in \bZ_{>0}}.
$$
We then define the {\em $n$-th Weil-Betti number} $b^W_n(X)$ of $X$ to be the coefficient of $z^n$ in $P(X,z)$. 
\end{definition}
As remarked above, if $X$ is smooth and projective, then it has a Weil-Poincar\'e polynomial. If furthermore, $X$ is obtained by reduction modulo a prime of a ring of integers, then the Weil-Poincar\'e polynomial is the usual Poincar\'e polynomial of the corresponding complex variety $X(\bC)$, that is, its coefficients give the topological Betti numbers of $X(\bC)$. Unfortunately, we wish to apply this theory to the reduced Hilbert schemes of $\cA$, which unlike the case of smooth surfaces, may not be smooth in general. 

As in \cite{Gottsche}, we will be interested in the following ``Poincar\'e'' zeta function. 
\begin{definition}  \label{def:PoincareZeta}
Let $\cA$ be an order on a quasi-projective variety $X$ over $\bF_q$. Suppose that for all $n \in \bN$, the reduced Hilbert schemes $\Hilb^n_{red}(\cA)$ have Weil-Poincar\'e polynomials. Then we define
$$Z^{\text{Poin}}_{\cA}(z,t) := \sum_n P(\Hilb^n_{red}(\cA),z)t^n.
$$
\end{definition}

We will need some slight modifications of results from \cite{Gottsche}. Using the algebra homomorphism $\nu$ in (\ref{eq:defNu}), we easily obtain the following generalisation of \cite[Remark~1.2.2]{Gottsche}
\begin{corollary}  \label{cor:PointCountToPoincare}
Suppose $X_1,\ldots,X_l$ are $\bF_q$-varieties which have Weil-Poincar\'e polynomials. If $Y$ is another $\bF_q$-variety whose point counts are given by a polynomial $F$ over $\bQ$ in the variables, $u, v_{ij}$ for $1\leq i \leq l, \ \ 1 \leq j \leq m$ via 
\begin{equation}  \label{eq:PointCountToPoincare}
|Y(\bF_{q^n})| = F(q^n, |X_1(\bF_{q^n})|, \ldots, |X_1(\bF_{q^{nm}})|,|X_2(\bF_{q^n})|, \ldots ,|X_l(\bF_{q^{nm}})|)
\end{equation}
for all positive integers $n$, 
then $Y$ has a Weil-Poincar\'e polynomial determined by 
$$P(Y,-z) = F(z^2,P(X_1,-z), \ldots, P(X_1,-z^m),P(X_2, -z), \ldots, P(X_l, -z^m)).
$$
\end{corollary}
\begin{proof}
This follows on applying $\nu$ to both sides of (\ref{eq:PointCountToPoincare}) and noting that $|X_i(\bF_{q^{nj}})|$ is obtained from $|X_i(\bF_{q^{n}})|$ by applying the group algebra homomorphism $\bQ \cW \to \bQ \cW$ corresponding to the group homomorphism $\cW \to \cW \colon \alpha \mapsto \alpha^j$. 
\end{proof}

Suppose now that $\cA$ is a degree $d$ maximal order on a smooth quasi-projective $\bF_q$-variety $X$, where $H^0(\mathcal{O}_X)$ contains primitive $d$-th roots of unity and $q$ is relatively prime to $d$. Suppose further that all the ramification data satisfy the following
\begin{assumption}  \label{ass:terminalRam}
We will assume that 
\begin{enumerate}
    \item all the ramification indices are equal to $e$ for some $e$,
    \item the ramification curves are all smooth and cross normally and
    \item at all these normal crossings in the ramification locus, the ramification covers are totally ramified.
\end{enumerate}
\end{assumption}
These assumptions ensure that we are now in the situation of Proposition~\ref{prop:zetaSmoothRamGlobal}. We wish to calculate $Z_{\cA}^{\text{Poin}}(z,t)$. 

Let us first consider the special case where there is a single smooth ramification curve $Y \subset X$ and the ramification cover is given by an $e:1$ cover $\tilde{Y} \to Y$, necessarily \'etale. To simplify notation we let $r = d/e$.
In this case 
\begin{align*} 
Z^{\Hilb}_{\cA}(q,t) & := \sum_n | \Hilb_{red}^n(\cA)(\bF_{q})| t^n \\
& = \prod_{n=1}^{\infty} \left(
\prod_{j=1}^d Z^{\Serre}_X(q^{nd-j}t^{nd}) \cdot\prod_{j=1}^dZ^{\Serre}_Y(q^{nd-j}t^{nd})^{-1}\cdot\prod_{j=1}^r Z^{\Serre}_{\tilde{Y}}(q^{nr-j}t^{nr})
\right)
\end{align*}
by Propositions~\ref{prop:zetaAzumaya} and \ref{prop:zetaSmoothRamGlobal}. 

In general, where the ramification locus $Y$ is not necessarily a single curve (but still satisfies Assumption~\ref{ass:terminalRam}), then setting $U := X-Y$ and stratifying the ramification locus, we see there exist smooth quasi-projective curves $\tilde{Y}_1, \ldots, \tilde{Y}_h$ such that 

\begin{equation}  \label{eq:stratifyHilbP}
Z^{\Hilb}_{\cA}(q,t) = \prod_{n=1}^{\infty} \left(
\prod_{j=1}^d Z^{\Serre}_U(q^{nd-j}t^{nd}) \cdot\prod_{i=1}^h \prod_{j=1}^r Z^{\Serre}_{\tilde{Y}_i}(q^{nr-j}t^{nr})
\right). 
\end{equation}

Recall that 
$
Z^{\Serre}_X(t) = \exp\left( \sum_{k=1}^{\infty} \frac{|X(\bF_{q^k})|}{k}t^k\right)
$
so
\begin{multline}  \label{eq:logHilbP}
\log Z^{\Hilb}_{\cA}(q,t) = \sum_{n=1}^{\infty} \left(
\sum_{j=1}^d \sum_{k=1}^{\infty} \frac{|U(\bF_{q^k})|}{k}q^{knd-kj}t^{knd} + \sum_{i=1}^h\sum_{j=1}^r \sum_{k=1}^{\infty} \frac{|\tilde{Y}_i(\bF_{q^k})|}{k}q^{knr-kj}t^{knr}
\right) \\
= \sum_{k=1}^{\infty} \left(
\frac{|U(\bF_{q^k})|}{k}\frac{t^{kd}(q^{dk}-1)}{(1-(qt)^{kd})(q^k-1)} + 
\sum_{i=1}^h\frac{|\tilde{Y}_i(\bF_{q^k})|}{k}\frac{t^{kr}(q^{rk}-1)}{(1-(qt)^{kr})(q^k-1)}
\right).
\end{multline}

Formula~\ref{eq:logHilbP} also holds for $q$ replaced with any positive integer power, so we may exponentiate the right hand side and extract the coefficient of $t^m$ to see that $|\Hilb^m_{red}(\cA)(\bF_{q^n})|$ is given by a polynomial function over $\bQ$ in $q^n$ as well as the $|U(\bF_{q^{nk}})|,|\tilde{Y}_i(\bF_{q^{nk}})|$ for a finite number of $k$ (note that the fractions $\frac{q^{dk}-1}{q^k-1},\frac{q^{dr}-1}{q^k-1}$ are just polynomials in $q$). Since the $\tilde{Y}_i$ are smooth projective curves over $\mathbb{F}_q$, they have Weil-Poincar\'e polynomials by the Weil conjectures. In fact, the same is true for \textit{any} curve over $\mathbb{F}_q$. To see this, we first note that, for any closed subvariety $V'$ of any variety $V$ over $\mathbb{F}_q$, we have $|V(\mathbb{F}_{q^n})|=|V'(\mathbb{F}_{q^n})|+|V''(\mathbb{F}_{q^n})|$ where $V''=V-V'$. Then we apply the resolution of singularities for curves and the Weil conjectures to reduce to the obvious case of points. Now, to see that $U$ also has a Weil-Poincar\'e polynomial, we combine the resolution of singularities for surfaces with the Weil conjectures to reduce to the curve case (c.f. \cite[Remark 2.20]{mustata}). We may thus apply Corollary~
\ref{cor:PointCountToPoincare} to obtain

\begin{theorem}   \label{thm:PoincareZetaOrder}
Suppose $\cA$ is a degree $d$ maximal order on a smooth quasi-projective 
$\bF_q$-variety $X$, where $H^0(\cO_X)$ contains a primitive $d$-th root of unity and $q$ is relatively prime to $d$. Suppose that the ramification data of $\cA$ satisfies Assumption~\ref{ass:terminalRam} so that there are smooth curves $\tilde{Y}_1,\ldots, \tilde{Y}_h$ as above giving ramification covers of degree $e$. Let $r = d/e$ and $U$ be the Azumaya locus. Then the reduced Hilbert schemes $\Hilb^n_{red}(\cA)$ have Weil-Poincar\'e polynomials and 
$$ Z^{\text{Poin}}_{\cA}(z,t) = Z^{\text{Poin}}_{\cA,U}(z,t)Z^{\text{Poin}}_{\cA,\tilde{Y}_1}(z,t)\ldots Z^{\text{Poin}}_{\cA,\tilde{Y}_h}(z,t)
$$
where
\begin{align*}
Z^{\text{Poin}}_{\cA,U}(-z,t) & = \exp\left( \sum_{k=1}^{\infty} \frac{P(U,-z^k)}{k}\frac{t^{kd}(z^{2dk}-1)}{(1-(z^2t)^{kd})(z^{2k}-1)}\right) 
\\
Z^{\text{Poin}}_{\cA,\tilde{Y}_i}(-z,t) & = \exp\left( \sum_{k=1}^{\infty} \frac{P(\tilde{Y}_i,-z^k)}{k}\frac{t^{kr}(z^{2rk}-1)}{(1-(z^2t)^{kr})(z^{2k}-1)}\right). 
\end{align*}
\end{theorem}

As usual, the formulae above can be written in terms of the Weil-Betti numbers

\begin{addendum}  \label{add:PoincareZetaOrder}
\begin{align*}
Z^{\text{Poin}}_{\cA,U}(z,t) & = 
\prod_{i=0}^4 \prod_{j=0}^{\infty} \prod_{l=0}^{d-1} (1-(-z)^{2jd+2l+i}t^{jd+d})^{-(-1)^i b^W_i(U)}
\\
Z^{\text{Poin}}_{\cA,\tilde{Y}_i}(z,t) & = \prod_{j=0}^{\infty} \prod_{l=0}^{r-1} \frac{(1+z^{2jr+2l+1}t^{jr+r})^{b^W_1(\tilde{Y}_i)}}{(1-z^{2jr+2l}t^{jr+r})^{b^W_0(\tilde{Y}_i)}(1-z^{2jr+2l+2}t^{jr+r})^{b^W_2(\tilde{Y}_i)}}.
\end{align*}
\end{addendum}

\section{Zeta functions for noncommutative projective planes}  \label{sec:eg}

In this section, we compute various zeta functions for noncommutative analogues of the projective plane. These should certainly include Azumaya algebras on the projective plane, but following the work of Artin-Tate-van den Bergh \cite{ATV1990}, one should also include maximal orders on $\mathbb{P}^2$ ramified on cubic curves. These include all the orders Morita equivalent to the three-dimensional Sklyanin algebra. They also appear naturally in the minimal model program for orders on surfaces \cite{CI05}. We will review their ramification below.


Before looking at ``Sklyanin orders'' and their cousins, we consider the situation of Azumaya algebras.
\begin{example}[Azumaya algebras]
Let $\cA$ be an Azumaya algebra of degree $d$ over a smooth projective surface $X$ over $\bF_q$. The Poincar\'e zeta function is given in Addendum~\ref{add:PoincareZetaOrder}. If $e(X)$ is the Euler characteristic of $X$, then we have
$$
Z^{\text{Poin}}_{\cA}(-1,t) = 
\prod_{j=0}^{\infty} (1 - t^{dj+d})^{-e(X)d}.
$$
When $d=1$ and $t=e^{2\pi i\tau}$, we obtain G\"ottsche's formula
\begin{align}\label{got_eul}
   Z^{\text{Poin}}_{\cA}(-1,t) = 
\prod_{j=0}^{\infty} (1 - t^{j+1})^{-e(X)}=\left(\frac{t^{1/24}}{\eta(\tau)}\right)^{e(X)} 
\end{align}
involving the Dedekind eta function $\eta(\tau)$, a modular form of weight $1/2$ \cite[Remark~2.3.12]{Gottsche}. Modularity of ``Euler" zeta functions for singular surfaces is an active research topic, with some bearing on general conjectures from physics about invariants of moduli spaces \cite{Gottsche_2002}. Bryan and Gyenge \cite{BG} have some nice results in this direction.
\end{example}

We now turn our attention to the non-Azumaya case. As motivation, note that a very natural candidate for a noncommutative affine plane over a field $\bF$ results from the skew polynomial ring $\bF_{\xi}[x,y]$ where $\xi \in \bF^\times$ gives the skew-commutative relation $yx = \xi xy$. When $\xi$ is a primitive $e$-th root of unity, then this is an order over its centre $R = \bF[u:=x^e,v:=y^e]$ which is Azumaya when $uv \neq 0$. The ramification above the coordinate lines $u=0, v=0$ are given respectively by the covers $\bF[x]$ and $\bF[y]$ which are totally ramified at the point where the ramification locus has a normal crossing point. Hence Assumption~\ref{ass:terminalRam} is satisfied. One can compactify this example to a maximal order on $\bP^2$ which also ramifies at the third coordinate line, and the ramification cover is similarly the degree $e$ cover that is totally ramified just at the coordinate points. In particular, the ramification locus here is a singular cubic consisting of 3 non-concurrent lines. More generally, the noncommutative projective planes given by three-dimensional regular algebras which are finite over their centre correspond to orders on $\bP^2$ which are ramified on cubic curves. 

\begin{example}  \label{eg:singRam}
Let $\bar{Y} \subset \bP^2_{\bar{\bF}_q}$ be a singular cubic curve which is either i) 3 non-concurrent lines, ii) a conic and a line meeting transversally or iii) a nodal cubic. Note that all the components of $\bar{Y}$ are rational curves. Then given integers $e>1, r \in \bZ_{>0}$ coprime to $q$, the Artin-Mumford sequence \cite{AM}, guarantees the existence of a maximal order $\bar{\cA}$ of degree $d=re$ that is ramified on $\bar{Y}$ and the ramification covers are the degree $e$ covers totally ramified at the singular points of $\bar{Y}$. By picking $q$ sufficiently large, we may assume that the order is defined over $\bF_q$. Suppose $\cA$ is such an order on $\bP^2_{\bF_q}$ which is ramified on a singular cubic curve $Y \subset \bP^2_{u,v,w}$. By enlarging $q$ if necessary, we may assume $\bF_q$ contains primitive $d$-th roots of unity and the components of $Y$ are geometrically irreducible and the singular points are all $\bF_q$-rational too. 

We first consider the case where $Y$ consists of 3 lines which we may as well take to be the coordinate lines. The ramification covers are all copies of $\bP^1_{\bF_q}$ which totally ramify above the coordinate points. We can express $Y$ as the disjoint union of a coordinate line, say $Y_1: u=0$, an affine coordinate line say $Y_2: v=0, u \neq 0$ and the rest $Y_3$. The ramification covers are respectively $\tilde{Y}_1 \simeq \bP^1 \simeq Y_1, \tilde{Y}_2 \simeq \bA^1 \simeq Y_2, \tilde{Y}_1 \simeq \bA^1 - \text{pt} \simeq Y_1$ (non-canonical isomorphisms). To compute the zeta functions of interest, it suffices to note that $Z^{\Serre}$ is multiplicative on disjoint unions,  $P(?,z)$ is additive and note 
\begin{align*}
    Z^{\Serre}_{\bP_1}(t) & = \frac{1}{(1-t)(1-qt)} & Z^{\Serre}_{\text{pt}}(t) & = \frac{1}{1-t} \\
    P(\bP^1,z) & = 1 + z^2 & P(\text{pt},z) & = 1.
\end{align*}
This gives 
\begin{align}
    Z^{\Serre}_{Y}(t) & = \frac{1}{(1-qt)^3} & P(Y,z) = 3z^2 
\end{align}
Propositions~\ref{prop:zetaAzumaya}, \ref{prop:zetaSmoothRamGlobal} then give 
\begin{equation}  \label{eq:zetaSingCubicRam}
    Z_{\cA}(t) = \prod_{n=1}^{\infty} \left(
\prod_{j=1}^d Z^{\Serre}_{\bP^2}(q^{nd-j}t^{nd}) \cdot\prod_{j=1}^d Z^{\Serre}_Y(q^{nd-j}t^{nd})^{-1}\cdot\prod_{j=1}^r Z^{\Serre}_{Y}(q^{nr-j}t^{nr})
\right).
\end{equation}
This gives Theorem~\ref{thm:introSklyanin} in this case. 
We can also compute the Poincar\'e zeta function, comparing it to the that of an Azumaya algebra $M_d(\cO_{\bP^2})$ of the same degree.
\begin{equation}  \label{eq:PoincareZetaSingCubicRam}
\frac{Z^{Poin}_{\cA}(z,t)}{Z^{Poin}_{M_d(\cO_{\bP^2})}(z,t)} = \prod_{j=0}^{\infty} \frac{\prod_{l=0}^{r-1} (1-z^{2jr+2l+2}t^{jr+r})^{-h}}{\prod_{l=0}^{d-1} (1-z^{2jd+2l+2}t^{jd+d})^{-h}}
\end{equation}
where $h=3$. 

When $Y$ consists of a conic and a line, the argument above can be repeated to give (\ref{eq:zetaSingCubicRam}) but with $Z^{\Serre}_Y(t) = \frac{1}{(1-qt)^2}$ and (\ref{eq:PoincareZetaSingCubicRam}) with $h=2$. Although strictly speaking, not covered by Assumption~\ref{ass:terminalRam}, the methods above do also apply to the case where the ramification is a nodal cubic. In this case, (\ref{eq:zetaSingCubicRam}) holds with $Z^{\Serre}_Y(t) = \frac{1}{1-qt}$ and (\ref{eq:PoincareZetaSingCubicRam}) with $h=1$. This concludes the proof of Theorem~\ref{thm:introSklyanin} when the ramification locus in singular. We conclude this example by noting that the ``Euler'' zeta function in this case is 
\begin{equation}  \label{eq:eulerZetaSkewPoly}
Z^{Poin}_{\cA}(-1,t) = 
\prod_{j=0}^{\infty} (1 - t^{dj+d})^{(h-3)d}\cdot
\prod_{j=0}^{\infty} (1 - t^{rj+r})^{-hr}.
\end{equation}
\end{example}

\begin{example}[Sklyanin orders]
Suppose that $\bF_q$ has primitive $e$-th roots of unity for some $e$. Let $\bar{Y} \subset \bP^2_{\bar{\bF}_q}$ be an elliptic curve and $\tilde{\bar{Y}}$ be an \'etale cyclic cover of $\bar{Y}$ of degree $e$. Then there exists a maximal order $\bar{\cA}$ on $\bP^2_{\bar{\bF}_q}$, ramified on $\bar{Y}$ with ramification given by the cover $\tilde{\bar{Y}} \to \bar{Y}$. Such an order is called a {\em Sklyanin order} since it has a homogeneous coordinate ring which is a three-dimensional Sklyanin algebra. Enlarging $q$ if necessary, we may assume that the maximal order is defined over $\bF_q$. There thus exist degree $e$ maximal orders $\cA$ on $\bP^2_{\bF_q}$, ramified on a smooth elliptic curve $Y$ with ramification given by a cyclic \'etale cover $\tilde{Y} \to Y$ of degree $e$. On enlarging $q$ further if necessary, we may assume $\bF_q$ is integrally closed in both $K(Y)$ and $K(\tilde{Y})$. Let $\cA$ be any maximal order on $\bP^2_{\bF_q}$ with such ramification data. By Weil's theorem \cite[Theorem~9.16B]{Rosen}, $Z^{\Serre}_{\tilde{Y}}(t) = Z^{\Serre}_Y(t) L(t)$ for some polynomial $L(t)$. Now both $Y$ and $\tilde{Y}$ have the same genus, so by the Weil conjectures \cite[Theorem~5.9]{Rosen}, the only possibility is that $L(t) = 1$. This gives the final case of Theorem~\ref{eq:introZeta}. The Poincar\'e zeta functions are readily computed from Theorem~\ref{thm:PoincareZetaOrder} and Addendum~\ref{add:PoincareZetaOrder} since the Betti numbers of $Y$ and $\tilde{Y}$ are $b_0 = 1 = b_2, b_1 = 2$. We only remark that for the ``Euler'' zeta function we have 
$$
Z^{\text{Poin}}_{\cA}(-1, t) = Z^{\text{Poin}}_{M_d(\cO_{\bP^2})}(-1, t).
$$
\end{example}

\begin{remark}
  Let $X$ be a smooth surface over a field $k$. Consider the class $\mathbb{L}=[\mathbb{A}^1]\in K_0(\Var_k)$ of the affine line in the Grothendieck ring of varieties over $k$, and consider Kapranov's \cite{kapranov} motivic zeta function $Z^{\mot}_X(t)\in 1+tK_0(\Var_k)[[ t]]$; see \cite{mustata} for a gentle introduction. Then we have the motivic G\"ottsche-Larsen-Lunts formula 
  $$
      \sum_{n=0}^\infty [\Hilb^n(X)] t^n=\prod_{m=1}^\infty Z^{\mot}_X(\mathbb{L}^{m-1}t^m).$$
  G\"ottsche \cite{Goettsche2000} proved this for algebraically closed fields $k$ of characteristic zero; Larsen and Lunts \cite{Larsen&Lunts} recently proved it for arbitrary $k$. Motivic formulae like this, that are uniform in $k$, provide a universal bridge between Euler characteristics with compact support ($k=\mathbb{C}$) and $\mathbb{F}_q$-point counts ($k=\mathbb{F}_q$). Specialising to the former, so that each $Z^{\mot}_X(\bL^{m-1}t^m)$ becomes $(1-t^m)^{-e(X)}$, we recover G\"ottsche's ``Euler" formula (\ref{got_eul}). Specialising to the latter, so that each $Z^{\mot}_X(\bL^{m-1}t^m)$ becomes $Z^{\Serre}_X(q^{m-1}t^m)$, we recover the $d=1$ case of our global Azumaya formula (\ref{eq:zetaAzumaya}). It would be interesting to reproduce our noncommutative results in this vein.
\end{remark}

\bibliographystyle{amsalpha}

\bibliography{references}

\end{document}